\documentclass{article}
\usepackage{fixmath}

\usepackage[T1]{fontenc}
\usepackage{ucs}
\usepackage[utf8x]{inputenc}
\usepackage[english]{babel}

\usepackage{microtype}
\usepackage{babel}

\usepackage{authblk}

\usepackage{mathtools} 

\usepackage{enumerate}

\usepackage{geometry} 

\usepackage{amsfonts, amsmath, amsthm, amssymb}
\usepackage{mathrsfs}

\usepackage{hyperref} 

\usepackage[usenames,dvipsnames]{xcolor}

\newtheorem{theorem}{Theorem}

\newtheorem{lemma}{Lemma}

\newtheorem{proposition}{Proposition}
\newtheorem{corollary}{Corollary}

\theoremstyle{definition}

\theoremstyle{remark}
\newtheorem{remark}{Remark}

\DeclareMathOperator{\Cyc}{Cyc}

\DeclareMathOperator{\Cov}{Cov}

\providecommand{\Exc}[2]{\mathbb{E}\left[#1\middle| #2\right]}

\providecommand{\card}[1]{\texttt{\#}#1}
\providecommand{\inorm}[1]{\lVert#1\rVert}

\newcommand{\Ex}{\mathbb{E}}
\providecommand{\Exc}[2]{\mathbb E\left[\right]#1\middle| #2\right]}

\providecommand{\abs}[1]{\lvert#1\rvert}

\providecommand{\risfac}[2]{#1^{\overline{#2}}} 
\DeclareRobustCommand{\stirling}{\genfrac{[}{]}{0pt}{}}

\newcommand{\conv}{\mathop{\scalebox{1.5}{\raisebox{-0.2ex}{$\ast$}}}} 

\providecommand{\Prob}[1]{\mathbb{P}\left\{#1\right\}}
\providecommand{\indicator}[1]{{\mathbf 1}_{\{#1\}}}

\providecommand{\keywords}[1]
{
	\small	
	\textbf{\textit{Keywords---}} #1
}

\title{Occupation times and areas derived from random sampling}
\author[1]{Frank Aurzada}
\author[2]{Leif D\"oring}
\author[2]{Helmut H.~Pitters}
\affil[1]{Department of Mathematics, Technical University of Darmstadt}
\affil[2]{Mathematics Institute, University of Mannheim}

\begin{document}

\maketitle

\begin{abstract}
We consider the occupation area of spherical (fractional) Brownian motion, i.e.\ the area where the process is positive, and show that it is uniformly distributed. For the proof, we introduce a new simple combinatorial view on occupation times of stochastic processes that turns out to be surprisingly effective. A sampling method is used to relate the moments of occupation times to persistence probabilities of random walks that again relate to combinatorial factors in the moments of beta distributions. Our approach also yields a new and completely elementary proof of L\'evy's second arcsine law for Brownian motion. Further, combined with Spitzer's formula and the use of Bell polynomials,  we give a characterisation of the distribution of the occupation times for all L\'evy processes.
\end{abstract}

\bigskip

\section{Introduction and main results}

Consider a measure space $(I, \mathcal I, \alpha)$, where $\alpha$ is a finite measure with total mass $|\alpha|=\alpha(I)$, and a stochastic process $X=\{ X_t, t\in I\}$ with index set $I$ whose state space $\mathscr X$ is endowed with some sigma algebra $\mathcal X$. We do not assume $I$ to be an ordered set. For a real-valued, non-negative, measurable function $f:\mathscr X\to [0, \infty)$ consider the path integral
$$
  \int_I f(X_s)\alpha(ds).
$$

Path integrals for diverse stochastic processes have a rich history in several areas of probability theory. In the present article, we deal with occupation times $\int_I \indicator{X_s\in S}\alpha(ds)$ for some measurable set $S$.  For $I=[0,t]$, $\alpha$ the Lebesgue measure, and $S$ measurable, this is the portion of time that the process spends in the set $S$. Most classically, the occupation time of the non-negative half-line $S=[0,\infty)$ during $[0, 1]$ by a Brownian motion is well-known to be arcsine distributed, i.e.\ it has the density $\pi^{-1} (x(1-x))^{-1/2}$ on $(0,1)$. This result goes back to Paul L\'evy~\cite{Levy1939} and is sometimes referred to as the second arcsine law for Brownian motion, cf.~\cite{MoertersPeres2010}. Since L\'evy's seminal work, many proofs for this result have been found (e.g.~Kac's derivation via the Feynman-Kac formula as expounded in~\cite[application of Theorem 7.43]{MoertersPeres2010}, or via approximation by (simple) random walks, cf.~\cite[Theorem 5.28]{MoertersPeres2010}). Further, various generalizations to other processes have been considered (see for instance~\cite{KallianpurRobbins1953, DarlingKac1957, Bingham1971, FellerVol1, GetoorSharpe1994, Knight1996, BinghamDoney1988, KhoshnevisanPemantle2004, ErnstShepp2017, Mountford1990, MeyreWerner1995}).\smallskip


While the one-dimensional stochastic process setting is well-understood, many open problems remain for multi-dimensional processes and processes with general index sets. Most prominently, characterising the distribution of occupation times of planar (and higher dimensional) Brownian motion (random walks) in cones are open problems to this day (except for cases that may trivially be reduced to one-dimensional problems). The major focus of this paper is on random fields, i.e.\ on processes with multidimensional index sets. Here, the Brownian sheet is a natural object to consider, and \cite{KhoshnevisanPemantle2004} derives asymptotic bounds, but the exact distribution of the occupation `area' of the Brownian sheet remains unknown. For the Brownian pillow we refer to \cite{Hashorva}. In this paper we compute the distribution of the occupation area for the fractional generalisation of L\'evy's spherical Brownian motion. This is our main result. \smallskip 

To motivate our approach let us recall some attempts towards the occupation times of planar Brownian motion using moments. Note that occupation times are bounded random variables and as such are uniquely determined by their moments. As a specific example consider the time that planar Brownian motion spends in some fixed cone $C$. The problem to characterise the distribution of this time was put forth in \cite{BinghamDoney1988} and is still open. The authors were able to derive the first three moments of this occupation time if $C$ is taken to be a quadrant. Motivated by this work,~\cite{ErnstShepp2017} studied the time that planar Brownian motion spends in the `hourglass', i.e.\ the union of the first and third quadrant, and rephrased this problem in the language of Kontorovich-Lebedev transforms.  Desbois~\cite{Desbois2007} generalized the quadrant problem to wedges with apex at the origin and some angle $\theta>0$.  Employing methods from physics, the author computed the first three moments in the case of a wedge with angle $\theta$, the fourth moment in the quadrant case ($\theta=\pi/2$), and derived a general formula for second moments in high-dimensional orthants. We follow these research efforts and attack occupation time distributions through their integer moments, introducing a simple sampling method.\smallskip

Suppose that we were to `guess' the proportion of time that the process $X$ spends in some set $S$ during $[0, t]$, and to this end we were allowed to sample $X$ at $m$ instances chosen according to our liking. It seems rather natural to choose the times $U_1, \ldots, U_m$ independently (and independent of $X$) and uniformly at random in $[0, t],$ and to take the empirical probability $\card \{1\leq k\leq m\colon X_{U_k}\in S\}/m$ as an estimator of said proportion. In fact, it turns out that the probability that $X$ is in $S$ at all times $U_1, \ldots, U_m$ agrees with the $m$-th moment of the occupation time of $S$ (up to the factor $t^m$), a generalization of which we will see in Proposition~\ref{prop:sampling}. Sampling a stochastic process at random times is by no means a new idea, and has been employed in various other contexts. For instance, the random tree may be constructed from broken lines derived from Brownian excursion sampled at independent uniform times~\cite{Aldous1993, LeGall1993}. In~\cite{Pitman1999} the author studied Brownian motion, bridge, excursion and meander by sampling at i.i.d.~uniform times, and the convex hull of multidimensional Brownian motion was studied in~\cite{Eldan2014} by sampling at the points of an independent Poisson process.\smallskip

%
%

A surprising consequence of the computation of moments by means of sampling at random times is a completely elementary proof for the arcsine law of the Brownian motion, the uniform distribution of the occupation time of L\'evy bridges, and also a new characterisation of the occupation times for all L\'evy processes. Our approach combines occupation time moments with random walk probabilities and elementary combinatorics. The use of combinatorics is not surprising as beta distributions often appear as occupation time distributions have explicit moment expressions involving elementary combinatorial factors. For example, the $m$-th moments of the arcsine distribution are $2^{-2m}\binom{2m}{m}$, combinatorial factors that appear in many combinatorial problems, in particular in persistence probabilities of random walks. This suggests to ask if the $m$-th moments of occupation times are inherently related to combinatorial terms. Our answer is yes. The main insight of this article is to realise that the following simple sampling formula is a surprisingly effective link to relate occupation times, random walks, and, depending on the situation, beta distributions.

\begin{proposition}
\label{prop:sampling}
Consider a stochastic process $(X_t)_{t\in I}$ indexed by a measure space $(I, \mathcal I, \alpha)$ that attains values in a measurable space $(\mathscr X,\mathcal X)$. Let $S\in\mathcal X$ and $m\in\mathbb N$.  Then 
   \begin{align} \label{eqn:samplingindicator}
      \Ex\left[\left(\int_I \indicator{X_t\in S}\alpha(dt)\right)^m\right]=\abs{\alpha}^m\Prob{ X_{U_1}\in S, \ldots, X_{U_m}\in S},\quad m\in\mathbb N,
    \end{align}
  where $U_1, U_2, \ldots$ is an i.i.d.~sequence independent of $X$ such that $U_1$ has distribution $\alpha/\abs\alpha$. 
  \end{proposition}
  \begin{proof}
      Set $f(x):=\indicator{x\in S}$. Re-writing the expectation w.r.t.\ to the distribution $\alpha/|\alpha|$ (independent of $X$) as integrals, we obtain
      $$
      \Ex\big[f(X_{U_1})\cdots f(X_{U_m})\big] =  \Ex\left[ \int_I\cdots \int_I f(X_{u_1}) \cdots f(X_{u_m}) \frac{\alpha(d u_1)}{|\alpha|} \cdots \frac{\alpha(d u_m)}{|\alpha|}\right].
      $$
      Multiplying by $|\alpha|^m$, noticing that all the integrals are identical, and inserting  $f(x):=\indicator{x\in S}$ shows the claim.
  \end{proof}


In order to discuss the use of this result, let us consider the example $I=[0,t]$, $\alpha$ the Lebesgue measure, and $S=[0,\infty)$. Then Proposition~\ref{prop:sampling} shows that the occupation time of continuous-time processes $(X_t)$ can be characterised through the persistence probabilities $\Prob{X_{U_1}\geq 0,...,X_{U_m}\geq 0}$. In many situations these persistence probabilities may be reduced to persistence probabilities $\Prob{S_1\geq 0,...,S_m\geq 0}$ for a well-understood discrete-time process $(S_n)$. For example for random walks, there is a vast literature going back to seminal works of Spitzer \cite{Spitzer1956} and Sparre Andersen \cite{Andersen1953b}, see also the exposition in~\cite[Section 1.3]{KabluchkoVysotskyZaporozhets2019} for more recent results, where such probabilities were computed under different assumptions on the set $S$. Little suprisingly, the moments of arcsine distributions appear naturally in persistence probabilities. While there is a long tradition of deriving arcsine laws for continuous-time processes from discrete-time processes using Donsker-type limiting arguments, the simple connection between moments of occupation times and persistence probabilities seems to be new.

\begin{remark}
The sampling approach also shows that the $m$-th moment of the occupation time of $d$-dimensional Brownian motion in some cone $C$ is equal to the probability that a $d$-dimensional random walk stays in the cone $C$ up to time $m.$ The exit time from a cone of a multi-dimensional random walk has received great interest in mathematical research (cf.~e.g.~\cite{GarbitRaschel2016}), not least because this quantity has connections to many areas such as representation theory~\cite{Biane1991, Biane1992}, conditioned random walks~\cite{Biane1991, Biane1992}, random matrices~\cite{Dyson1962}, non-colliding random walks~\cite{DenisovWachtel2010, EichelsbacherKoenig2008}, and enumerative combinatorics~\cite{Bousquet-MellouMishna2010, FayolleRaschel2012, JohnsonMishnaYeats2018}. We leave for future research whether this direct link between continuous-time occupation times and discrete-time exit probabilities may help to solve open problems for the planar and multidimensional Brownian motion. 
\end{remark}

\smallskip
\textbf{Organisation of the article:} In the following sections we illustrate the power of this simple approach. The paper is structured as follows. In Section~\ref{sec:arcsine} we  give a very simple proof of the second arcsine law of Brownian motion. In Section~\ref{sec:sfBM} we discuss the main result of this paper, i.e.\ we determine the distribution of the occupation `area' of L\'evy's Brownian motion on the sphere. In Section~\ref{sec:levysection} we characterise all occupation times of one-dimensional L\'evy processes using combinatorial expressions. The proofs are given in Section~\ref{sec:proofs}.

\subsection{An elementary proof of L\'evy's arcsine law} \label{sec:arcsine}

As a first illustration of our line of attack we give a new, very elementary proof of L\'evy's second arcsine law for Brownian motion.
\begin{theorem}[L\'evy~\cite{Levy1939}] \label{thm:levysfirstarcsinelaw}
    If $B$ is a standard Brownian motion, then $t^{-1}\int_0^t \indicator{B_s>0}ds$ is arcsine distributed.
\end{theorem}
In contrast to other proofs of the second arcsine law of Brownian motion, our proof is completely elementary and in particular does not require any limiting procedure nor does it employ analytic computations or excursion theory, as L\'evy's original proof. At first sight our argument might resemble proofs that approximate Brownian motion using discrete-time random walks. However, instead, we use an entirely different connection between Brownian motion and the so-called Laplace random walk. Instead of discretising $(B_t)$ and studying the same problem for random walks, the sampling method relates the moments of the occupation time of continuous Brownian motion to discrete persistence probabilities.

\begin{proof}[A simple proof of Theorem~\ref{thm:levysfirstarcsinelaw}]
    W.l.o.g.\ we may assume $t=1$, by the self-similiarity of Brownian motion. Fix $m\in\mathbb N$. The sampling formula (\ref{eqn:samplingindicator}) gives
\begin{equation} \label{eqn:samplingideabm1}
  \Ex \left [\left (\int_0^1 \indicator{  B_t>0 }dt \right )^m\right ] = \Prob{ B_{U_1}>0, \ldots, B_{U_m}>0  } = \Prob{ B_{U_{m:1}}>0, \ldots, B_{U_{m:m}}>0  },
\end{equation}
where $(U_i)$ are i.i.d.\ uniform in $[0,1]$ independent of the Brownian motion and $(U_{m:i})$ is the corresponding order statistics. Further, let $(E_i)$ be i.i.d.\ standard exponential random variables independent of the $(U_i)$ and of the Brownian motion and set $T_k:=\sum_{i=1}^k E_i$, $k=0,1,2,\ldots$. Conditioning on $T_{m+1}$ and on the $(U_i)$ (which are independent of the Brownian motion $B$), we can use the self-similarity of Brownian motion, $(B_{s})_{s\geq 0}=_d (T_{m+1}^{-1/2} B_{T_{m+1}s})_{s\geq 0}$, to see that 
the probability in (\ref{eqn:samplingideabm1}) equals
\begin{equation} \label{eqn:samplingideabm2}
\Prob{ B_{T_{m+1} U_{m:1}}>0, \ldots, B_{T_{m+1} U_{m:m}}>0  } =  \Prob{ B_{T_1}>0, \ldots, B_{T_m}>0  },
\end{equation}
where we used the independence of $(U_i)$ and $(E_i)$ from the Brownian motion and the fact that the vector $(T_{m+1} U_{m:1},\ldots,T_{m+1} U_{m:m})$ has the same distribution as $(T_1,\ldots,T_m)$, see e.g.\ Theorem V.2.2 in \cite{Devroye1986}.\smallskip

Thus, the moments of the occupation time of Brownian motion on the left-hand side in  (\ref{eqn:samplingideabm1})  are given by the persistence probabilities on the right-hand side in (\ref{eqn:samplingideabm2}). We note that these are the persistence probabilities of the Laplace random walk $R_i:=B_{T_i}$, $i=0,1,2,\ldots$. It is well-known that the probabilities on the right-hand side in (\ref{eqn:samplingideabm2}) are equal to $2^{-2m}\binom{2m}{m}$, which are -- in turn --  the moments of the arcsine distribution. Since the occupation times are bounded the proof of the second arcsine law of Brownian motion is complete.\smallskip

To keep the proof self-contained let us also give an elementary argument for the persistence probabilities in (\ref{eqn:samplingideabm2}). Define $\tau:=\min\{ j\in\{ 0,\ldots, m\} : R_j=\max_{k\in\{0,\ldots,m\}} R_k \}$ to be the first (and only) index where the maximum of $(R_k)_{k=0}^m$ is attained. 
Since $\tau\in\{0,\ldots,m\}$ by construction, we must have (using the continuity of the distribution of the $R_k$ in the second step):
\begin{align*}
1 =& \sum_{j=0}^m \Prob{ \tau = j} = \sum_{j=0}^m \Prob{ R_k < R_j, k=0,\ldots, j-1,j+1,\ldots, m} 
\\
=& \sum_{j=0}^m \Prob{ R_k < R_j, k=0,\ldots, j-1} \cdot \Prob{ R_k < R_j, k=j+1,\ldots, m} 
\\
=& \sum_{j=0}^m \Prob{ R_k>0, k=1,\ldots, j} \cdot \Prob{ R_k>0, k=1,\ldots, m-j},
\end{align*}
where we used the independence of increments of $(R_k)$ in the third step and the stationarity and the symmetry of the increments of $(R_k)$ in the fourth step. It is again elementary to show that the unique solution of this recursive equation is given by $\Prob{ R_k>0, k=1,\ldots, j} = \frac{(2j-1)!!}{(2j)!!} =  2^{-2j}\binom{2j}{j}$ for all $j=0,\ldots,m$.
To see the latter, multiply the recursion by $x\in[0,1)$, sum in $m$, and the generating function of the probabilities in question is found to be $(1-x)^{-1/2}$, cf.\ \cite{dembodinggao2013} for similar arguments.
\end{proof}
    The proof does not fully use the Brownian properties, in particular, continuity does not play a role in the sampling. Actually, exactly the same argument works for symmetric strictly stable L\'evy processes, recovering the arcsine law first derived by Kac~\cite{Kac1951}. Below we also provide a simple proof for the occupation time of a Brownian bridge to be $\mathcal U([0,1])$ but we do so directly in the more general setting of L\'evy bridges, cf.\ Theorem~\ref{thm:occupation_time_levy_bridge}.

\subsection{Spherical fractional Brownian motion}\label{sec:sfBM}
We now come to the main result of this article, the occupation `area' law for L\'evy's spherical Brownian motion and the fractional generalisation. Fix $H\in (0, 1/2]$ and $d\in\mathbb N$, $d\geq 2$, and let $\inorm x\coloneqq\sqrt{x_1^2+\cdots+x_d^2}$ denote the Euclidean norm of $x\in\mathbb R^d$. Recall that spherical fractional Brownian motion (spherical fBM) $X\coloneqq (X_t)_{t\in\mathbb S^{d-1}}$ is a centred Gaussian process on the unit $(d-1)$-sphere $\mathbb S^{d-1}\coloneqq \{ x\in\mathbb R^d\colon \inorm x=1 \}$ such that $X_O=0$ a.s.~for some arbitrary fixed point $O\in\mathbb S^{d-1}$ with 
\begin{align}\label{eq:sfbm_increment}
  \Ex[(X_s-X_t)^2] = (d(s, t))^{2H},\qquad s, t\in \mathbb S^{d-1},
\end{align}
where $d(s, t)$ denotes the geodesic distance between two points $s, t$ on $\mathbb S^{d-1}$. The special case $H=1/2$ was first studied by Paul L\'evy~\cite{Levy1965} and is sometimes referred to as L\'evy's spherical Brownian motion.  Istas~\cite{Istas2005} showed that there exists a Gaussian process indexed by $\mathbb S^{d-1}$ with covariance structure as in~\eqref{eq:sfbm_increment} if and only if $H\leq 1/2$. Let
\begin{align*}
  A\coloneqq\int_{\mathbb S^d} \indicator{X_s>0}\sigma^{d-1}(ds)
\end{align*}
denote the `area' that $X$ spends positive, or rather the measure of the area on $\mathbb S^{d-1}$ on which $X$ is positive as measured by the surface measure $\sigma^{d-1}$.
\smallskip

\begin{theorem}[Occupation time of spherical fractional Brownian motion]\label{thm:sfbm_sojourn_time}
Let $(X_t)_{t\in \mathbb S^{d-1}}$ be a spherical fractional Brownian motion $X$ with Hurst parameter $H\in (0, 1/2]$. Then
$$
  \abs{\sigma^{d-1}}^{-1}\int_{\mathbf S^{d-1}} \indicator{ X_s >0 }  \sigma^{d-1}(ds),
$$
i.e.\ the `area' that $X$ spends positive, is uniformly distributed on $(0, 1)$, where $\abs{\sigma^{d-1}}=\sigma^{d-1}(\mathbb S^{d-1})=2\pi^{\frac d 2}/\Gamma(\frac d 2)$ is the surface area of the unit $(d-1)$-sphere.
\end{theorem}

\subsection{L\'evy processes and bridges} \label{sec:levysection}
In this section, we apply the sampling formula to compute all moments of occupation times of one-dimensional L\'evy processes, i.e.\ stochastic processes with independent and stationary increments. Let $(X_t)_{t\geq 0}$ be a L\'evy process. We characterize the distribution of the random variable 
\begin{equation} \label{eqn:levyoccupation}
    A_t\coloneqq \int_0^t \indicator{X_s>0 }\, ds
\end{equation} by working out explicitly all its moments. In order to state the result, let us introduce some further notation. A partition of a set $S$ is a set, $\rho$ say, of nonempty pairwise disjoint subsets of $S$ whose union is $S$. The members of $\rho$ are also called the blocks of $\rho$. Let $\card S$ denote the cardinality of $S$ and for some natural number $n$ let ${\mathscr P}_n$ denote the set of all partitions of $\{1, \ldots, n\}$.
Further, we recall that $(f\ast g)(t):=\int_0^t f(t-s) g(s) d s$ is the convolution of two functions $f, g : [0,\infty)\to \mathbb R$. Sampling the occupations at Poisson times in combination with Spitzer's identity and a Bell polynomial trick yields the following moment formula:
\begin{theorem}[Occupation time of a L\'evy process]\label{thm:occupation_time}
Fix $m\geq 1$ arbitrarily. The $m$-th moment of the occupation time $A_t$ of the real-valued L\'evy process $X$ in the set $(0,\infty)$ is given by
  \begin{align}\label{eq:occupation_time}
    \Ex[A_t^m] &= \sum_{\rho\in{{\mathscr P}}_m}  \int_0^t \conv_{B\in\rho} \left (u^{\card B-1}\Prob{X_u>0}\right )(s)ds.
  \end{align}
In particular, the first two moments of $A_t$ are given by
\begin{align}\label{eq:occupation_time_moments}
  \Ex[A_t] &= \int_0^t \Prob{X_{s}>0} d s,\\\notag
  \Ex[A_t^2] &=\int_0^t s\Prob{X_{s}>0}ds+\int_0^t\int_0^s \Prob{X_u>0}\Prob{X_{s-u}>0}duds.
\end{align}
Equations~\eqref{eq:occupation_time} and~\eqref{eq:occupation_time_moments} still hold when their (strict) inequalities together with the (strict) inequality in the definition of the occupation time (\ref{eqn:levyoccupation}) are replaced by weak inequalities.
\end{theorem}
Theorem~\ref{thm:occupation_time} shows how to work out explicitly the moments of  the distribution of the occupation time above zero of a L\'evy process $X$. In particular, the formula shows that the distribution of $A_t$ is completely determined by the positivity function $s\mapsto \Prob{X_s>0}$. In fact, the only ingredient coming from the L\'evy process in the moment formula (\ref{eq:occupation_time}) is the positivity function. Equivalently, the theorem shows that the first moment of the occupation times already determines their entire distribution.\smallskip

There are a few situations in which the moments formulas can be used to compute the occupation time distributions. One example, that could not be treated in the literature before, is the $\frac{1}{2}$-stable subordinator with negative drift $\mu$ for which the positivity function is known as $\Prob{X_{t}>0}=\text{erf}(\sqrt{t/(4\mu)})$. The slightly tedious computations will be presented in an accompanying article. A more common situation is that of constant positivity, i.e.\ $\Prob{X_t>0}=c$ for all $t>0$, which occurs for instance in the case of strictly stable L\'evy processes. Inserting  into (\ref{eq:occupation_time}) leaves us with a simple  combinatorial expression for the moments of $A_t$. A short computation shows that those expressions are precisely those of the generalised arcsine distributions, i.e.\ a beta distribution with parameters $(a, b)=(c, 1-c)$ for some $c\in (0, 1)$. 

\begin{corollary}[cf.~\cite{GetoorSharpe1994}]\label{cor:GS} Fix $c\in(0,1)$.
The following two statements are equivalent:
\begin{enumerate}
    \item We have $\Prob{X_t>0}=c$ for all $t>0$.
    \item The occupation time $t^{-1} A_t=t^{-1} \int_0^t \indicator{ X_s>0 } ds$ is generalised arcsine distributed with parameter $c\in(0,1)$ for all $t>0$.
\end{enumerate}
\end{corollary}
The symmetric case $c=\frac{1}{2}$ thus recovers the classical arcsine law. The corollary can be deduced from Theorem \ref{thm:occupation_time} with a short combinatorial computation because the moments of generalised arcsine distributions have the combinatorial form
 \begin{align*}
 	\frac{\Gamma(m+c)}{\Gamma(m+1)\Gamma(c)}=\frac{\risfac{c}{m}}{m!},
 \end{align*}
 where $\risfac{x}{m}\coloneqq x(x+1)\cdots (x+m-1)$ denotes the $m$-the rising factorial power of $x\in\mathbb R$, and the last identity is easily seen by induction. The corollary was already proved by Getoor and Sharpe~\cite{GetoorSharpe1994} by guessing of Laplace transforms. Our proof highlights once more the combinatorial nature behind occupation times seen through their moments.\smallskip

Finally, we use the sampling method to provide a simple proof of the uniformity of occupation times for L\'evy bridges.
\begin{theorem}[Occupation time of a L\'evy bridge; cf.~\cite{FitzsimmonsGetoor1995} and~\cite{Knight1996}]\label{thm:occupation_time_levy_bridge}
    Let $X$ denote a L\'evy process and consider the stochastic process $\mathring X\coloneqq ( \mathring X_t)_{t\in [0, 1]}$ defined by $\mathring X_t\coloneqq X_t-tX_1$, that we refer to as the L\'evy bridge induced by $X$. Provided that the distribution of $X_1$ has no atoms, the occupation time $\mathring A\coloneqq \int_0^1 \indicator{\mathring X_t>0} dt$ of the L\'evy bridge $\mathring X$ is uniformly distributed on $(0, 1)$.
\end{theorem}
    The result in Theorem~\ref{thm:occupation_time_levy_bridge} essentially goes back to Fitzsimmons and Getoor~\cite{FitzsimmonsGetoor1995} and Knight~\cite{Knight1996}. In fact, Knight~\cite[Theorem 2.1(a)]{Knight1996} provides a complete characterization of L\'evy bridges with uniform occupation times. However, we consider our derivation interesting in its own right, as the  sampling approach yields significantly simpler proofs.

\section{Proofs of the theorems} \label{sec:proofs}

\subsection{Proofs for the spherical fractional Brownian motion result}

Before we start with the proof of our results on spherical fBM let us first examine its index set.
Looking at $\mathbb S^{d-1}$ through the glasses of Cartesian coordinates, there seems to be no natural way to order its elements that suits our purposes. Instead, the spherical coordinates naturally suggest an order on the sphere that is very useful.
Let us recall the definition of \emph{spherical coordinates.} For any point $x\in\mathbb R^{d}$ with Euclidean norm $r\coloneqq r(x)\coloneqq\inorm x$ its angles $(\varphi_1, \ldots, \varphi_{d-2}, \theta)\coloneqq (\varphi_1(x), \ldots, \varphi_{d-2}(x), \theta(x))\in [0, \pi)^{d-2}\times [0, 2\pi)$ are defined (cf.~\cite{Blumenson1960}) implicitly by
\begin{align}\label{eq:def_spherical_coordinates}
\begin{split}
  x_k &= r\cos\varphi_k\prod_{j=1}^{k-1}\sin\varphi_j,\qquad 1\leq k\leq d-2,\\
  x_{d-1} &= r\sin\theta\prod_{j=1}^{d-2}\sin\varphi_{j},\\
    \end{split}
\end{align}
\begin{align*}
\intertext{which then implies}
  x_d &= r\cos\theta\prod_{j=1}^{d-2} \sin\varphi_j.
\end{align*}
We refer to $(r(x), \varphi_1(x), \ldots, \varphi_{d-2}(x), \theta(x))$ as the spherical coordinates of $x$. At times we allow ourselves to slightly misuse terminology and refer to the angles \linebreak $(\varphi_1(x), \ldots, \varphi_{d-2}(x), \theta(x))$ as the spherical coordinates of $x$, in particular if $r(x)=1$. In what follows we agree on using the following (lexicographic) order $\leq$ on $\mathbb S^{d-1}$. For $x, x'\in\mathbb S^{d-1}$ we set $x\leq x'$ if one of the following (mutually exclusive) conditions holds:
\begin{enumerate}[i)]
  \item $x=x'$,
  \item $\theta(x)<\theta(x')$,
  \item $\theta(x)=\theta(x')$, and there exists $1\leq k\leq d-2$ such that $\varphi_j(x)=\varphi_j(x')$ for all $1\leq j\leq k-1$, and $\varphi_k(x)<\varphi_k(x')$.
\end{enumerate}
  In what follows we will deal with a finite number $U_1, \ldots, U_m$ of, say, i.i.d.~r.v.s sampled according to some \textit{continuous} distribution with support $\mathbb S^{d-1}$. Consequently, for any pair $U_i, U_j$ all their angles are distinct a.s. Therefore, their order, i.e.\ whether $U_i\leq U_j$ or $U_j\leq U_i$, is completely determined by $\theta(U_i)$ and $\theta(U_j)$. This means that the order statistics $U_{m:1}\leq\cdots\leq U_{m:m}$ again only depends on the angles $\theta(U_1), \ldots, \theta(U_m)$ a.s.\smallskip
  
  At the heart of our proof of Theorem~\ref{thm:sfbm_sojourn_time} lies the following proposition on the increments of spherical fBM that we consider of interest in its own right. We call a finite permutation $\pi$ a cyclic permutation if there is a decomposition of $\pi$ into one cycle only. For $m\in\mathbb N$ we denote by $\Cyc(m)$ the set of all cyclic permutations of $\{1, \ldots, m\}$. A finite sequence $(Y_1, \ldots, Y_m)$ of r.v.s is called cyclically exchangeable if for any cyclic permutation $\pi\in\Cyc(m)$ the random vectors $(Y_1, \ldots, Y_m)$ and $(Y_{\pi(1)}, \ldots, Y_{\pi(m)})$ have the same distribution. Intuition suggests that the increments of spherical fractional Brownian motion $X$ induced by the order statistics of $m$ i.i.d.~points $U_1, \ldots, U_m$ sampled from $\mathbb S^{d-1}$ uniformly at random should be cyclically exchangeable. (This is most easily seen first in the special case $d=2$.) Our next proposition shows that this intuition is in fact true.

\begin{proposition}\label{prop:sfbm_increments}
  Fix $H\in (0, 1/2]$, and let $(X_t)_{t\in\mathbb S^{d-1}}$ denote spherical fBM with Hurst index $H$ as defined by~\eqref{eq:sfbm_increment} with the property that $X_O=0$ a.s.~for some fixed (deterministic) $O \in\mathbb S^{d-1}$ with $\theta(O)=0$. Let $U_1, U_2, \ldots$ denote a sequence of i.i.d.~r.v.s uniformly distributed on $\mathbb S^{d-1}$. Fix $m\in\mathbb N$. Then the sequence of increments
  \begin{align}\label{eq:sfbm_increments}
    (X_{U_{m:k}}-X_{U_{{m:k-1}}})_{k=1}^{m+1}
  \end{align}
is cyclically exchangeable, where we set $U_{m:0}\coloneqq U_{m:m+1}\coloneqq O$.
\end{proposition}


Before we turn to the proof of Proposition~\ref{prop:sfbm_increments} we make some further observations.

\begin{lemma}\label{lem:uniform_sample_on_sphere}
  Let $U$ be a point sampled uniformly at random from $\mathbb S^{d-1}$. Then \linebreak $(\varphi_1(U), \ldots, \varphi_{d-2}(U))$ and $\theta(U)$ are independent, and $\theta(U)$ is uniformly distributed on $(0, 2\pi)$.
\end{lemma}

\begin{proof}
  Fix some arbitrary $x\in\mathbb R^d$. Notice from the definition of spherical coordinates in Equations~\eqref{eq:def_spherical_coordinates} that the angles of $x$ and $cx$ agree for any $c>0$, i.e.
    \begin{align*}
      \varphi_k(x) &= \varphi_k\left (cx\right ),\qquad 1\leq k\leq d-2,\quad \text{and}\quad
      \theta(x) = \theta\left (cx\right ).
    \end{align*}
  Moreover, the angles $\varphi_1, \ldots, \varphi_{d-2}$ depend on $x_1, \ldots, x_{d-2}$, but not on $x_{d-1}, x_d$. The projection of $x$ onto the hyperplane $x_1=\cdots=x_{d-2}=0$ has distance $\sqrt{x_{d-1}^2+x_d^2}=r\prod_{j=1}^{d-2}\sin\varphi_j$ from the (Euclidian) origin by the Pythagorean identity $\sin^2 \varphi+\cos^2 \varphi =1$, and since $\sin\varphi\geq 0$ for $\varphi\in [0, \pi)$. Consequently, $\sin\theta=x_{d-1}/\sqrt{x_{d-1}^2+x_d^2}$, and therefore $\theta$ only depends on $x_{d-1}$ and $x_d$. Recall now that $U=_d X/\inorm X$ with $X=(X_1, \ldots, X_d)$ having i.i.d.~standard Gaussian coordinates. This shows that $(\varphi_1(U), \ldots, \varphi_{d-2}(U))=_d (\varphi_1(X), \ldots, \varphi_{d-2}(X))$ and $\theta(U)=_d\theta(X)$ are independent. Moreover, $\theta$ is the angle enclosed by the positive $x_d$-axis and the line through the origin and the projection $(0, \ldots, 0, x_{d-1}, x_d)$ of $x$ onto the hyperplane $x_1=x_2=\cdots=x_{d-2}=0$. Since the distribution of $(X_{d-1}, X_d)$ is invariant under rotations in the plane, $\theta(X)$ is uniformly distributed on $(0, 2\pi)$.
\end{proof}

The last lemma allows us to show that the (geodesic) distances between consecutively ordered i.i.d~uniformly distributed points on $\mathbb S^{d-1}$ are exchangeable. Recall that a finite sequence $(Y_1, \ldots, Y_m)$ of r.v.s is called exchangeable if for any permutation $\pi$ the random vectors $(Y_1, \ldots, Y_m)$ and $(Y_{\pi(1)}, \ldots, Y_{\pi(m)})$ have the same distribution. Clearly, if $(Y_1, \ldots, Y_m)$ is exchangeable than it is also cyclically exchangeable.

\begin{proposition}\label{prop:cyclically_exchangeable_distances}
  Fix $m\in\mathbb N$. Let $U_1, \ldots, U_m$ be a sequence of i.i.d.~r.v.s with uniform distribution on $\mathbb S^{d-1}$. Then the random vector of geodesic distances
  \begin{align*}
    \left ( d(U_{m:k}, U_{m:k-1})\right )_{k=1}^{m+1}
  \end{align*}
between consecutive order statistics $U_{m:0}, U_{m:1}, \ldots, U_{m:m+1}$ is exchangeable, where $U_{m:0}\coloneqq U_{m:m+1}\coloneqq O\in\mathbb S^{d-1}$ is a fixed (deterministic) point with $\theta(O)=0$.
\end{proposition}

\begin{proof}
  We make use of the fact that the geodesic distance $d(x, x')$ between two points $x, x'\in\mathbb S^{d-1}$ satisfies $\cos d(x, x')= x\cdot x'$, where $x\cdot x'\coloneqq \sum_{k=1}^d x_kx'_k$ denotes the scalar product of $x$ and $x'$, cf.~\cite[p.~141--142]{Berger2010}. Thus it suffices to show that $(U_{m:k}\cdot U_{m:k-1})_{k=1}^{m+1}$ is exchangeable. Now, denoting by $(x)_\ell$ the $\ell$-th component of $x$, by definition of the scalar product,
  \begin{align*}
    U_{m:k}\cdot U_{m:k-1} =& \sum_{\ell=1}^d (U_{m:k})_\ell(U_{m:k-1})_\ell,
  \end{align*}
  and by the implicit definition of spherical coordinates, Equations~\eqref{eq:def_spherical_coordinates}, the last term equals    
    \begin{align*}
    =& \sum_{\ell=1}^{d-2} \cos\varphi_\ell(U_{m:k})\cos\varphi_\ell(U_{m:k-1})\prod_{j=1}^{\ell-1}\sin\varphi_j(U_{m:k})\sin\varphi_j(U_{m:k-1}) \\
    &+ \left ( \sin\theta(U_{m:k})\sin\theta(U_{m:k-1})+\cos\theta(U_{m:k})\cos\theta(U_{m:k-1})\right )\prod_{j=1}^{d-2}\sin\varphi_j(U_{m:k})\sin\varphi_j(U_{m:k-1})\\
    =& \sum_{\ell=1}^{d-2} \cos\varphi_\ell(U_{m:k})\cos\varphi_\ell(U_{m:k-1})\prod_{j=1}^{\ell-1}\sin\varphi_j(U_{m:k})\sin\varphi_j(U_{m:k-1})\\
      & + \cos(\theta(U_{m:k})-\theta(U_{m:k-1}))\prod_{j=1}^{d-2}\sin\varphi_j(U_{m:k})\sin\varphi_j(U_{m:k-1}),
  \end{align*}
  where we used the identity $\cos(\varphi-\varphi')=\cos \varphi\cos \varphi'+\sin \varphi\sin \varphi'$ in the last equation.
 By Lemma~\ref{lem:uniform_sample_on_sphere} $(\theta(U_{k}))_{k=1}^m$ is an i.i.d.\ sequence of uniform $(0, 2\pi)$ r.v.s. Consequently, the gaps $(\theta(U_{m:k})-\theta(U_{m:k-1}))_{k=1}^{m+1}$ rescaled by $1/(2\pi)$ obey a Dirichlet distribution with all parameters equal to one, cf.\ Theorem V.2.2 in \cite{Devroye1986}. In particular, the gaps are exchangeable. Since $(\varphi_1(U_{m:k}), \ldots, \varphi_{d-2}(U_{m:k}))_{k=1}^{m}$ is an i.i.d.\ sequence of random variables 
 and again by Lemma~\ref{lem:uniform_sample_on_sphere} and the expansion of the dot product in the last display we see that $(U_{m:k}\cdot U_{m:k-1})_{k=1}^m$ is exchangeable.
\end{proof}

We are now ready to show Proposition~\ref{prop:sfbm_increments}.

\begin{proof}[Proof of Proposition~\ref{prop:sfbm_increments}]
Set $d_H(s, t)\coloneqq (d(s, t))^{2H}$, and define the function \linebreak $\tilde c\colon (\mathbb S^{d-1})^4\to \mathbb R$ by 
\begin{align*}
  \tilde c (s, s', t, t') &\coloneqq \Cov(X_{s'}-X_s, X_{t'}-X_t) \\
  &= \Cov(X_{s'}, X_{t'})-\Cov(X_{s'}, X_{t})-\Cov(X_s, X_{t'})+\Cov(X_s, X_t)\\
  &= c(s', t')-c(s', t)-c(s, t')+c(s, t)\\
  &= \frac 1 2 \left (d_H(s', t)+d_H(s, t')-d_H(s', t')-d_H(s, t) \right ),
\end{align*}
where the covariance function of $X$,
\begin{align*}
  c(s, t) &= \frac{1}{2}(d_H(O, s)+d_H(O, t)-d_H(s, t)),\qquad s, t\in \mathbb S^{d-1},
\end{align*}
can be computed from~\eqref{eq:sfbm_increment}. Conditionally given $U\coloneqq (U_1, \ldots, U_m)$ the random vector of increments $(X_{U_{m:k}}-X_{U_{m:k-1}})_{k=1}^{m+1}$ has characteristic function
  \begin{align}
    \Ex\left [\exp \left (i\sum_{k=1}^{m+1} s_k \left (X_{U_{m:k}}-X_{U_{m:k-1}}\right )  \right )\Big|\, U  \right ] &= \exp\left ( -\frac{1}{2}s^\intercal Rs\right)\qquad (s\in\mathbb R^{m+1}),
  \end{align}
where $R=(R_{ij})$ is the $(m+1)\times (m+1)$ covariance matrix (a random matrix, as it depends on $(U_1,\ldots,U_m)$) defined by $R_{ij}\coloneqq \tilde c(U_{m:i}, U_{m:i-1}, U_{m:j}, U_{m:j-1})$.
Let $\pi\in\Cyc(m+1)$ be an arbitrary but fixed cyclic permutation. If we can show that the matrices $(R_{ij})_{i, j\in \{1,\ldots,m+1\}}$ and $(R_{\pi(i)\pi(j)})_{i, j\in \{1,\ldots,m+1\}}$ have the same distribution, the claim is proved. Assume without loss of generality that $i\leq j-1$. Notice that by definition of $\tilde c$
\begin{align*}
  &\quad R_{ij} \\
  &= \frac{1}{2}\left ( d_H(U_{m:i-1}, U_{m:j})+d_H(U_{m:i}, U_{m:j-1})-d_H(U_{m:i-1}, U_{m:j-1})-d_H(U_{m:i}, U_{m:j})\right)  \\
  &=_d \frac{1}{2}\bigg ( d_H(U_{m:\pi(i)-1}, U_{m:\pi(j)})+d_H(U_{m:\pi(i)}, U_{m:\pi(j)-1})\\
  &\qquad -d_H(U_{m:\pi(i)-1}, U_{m:\pi(j)-1}) -d_H(U_{m:\pi(i)}, U_{m:\pi(j)})\bigg )  \\
  &= R_{\pi(i)\pi(j)},
\end{align*}
where we applied Proposition~\ref{prop:cyclically_exchangeable_distances} in the second equality. The identity in the last math display implies
 that $(R_{ij})_{i, j\in \{1,\ldots,m+1\}}$ has the same distribution as $(R_{\pi(i)\pi(j)})_{i, j\in \{1,\ldots,m+1\}}$ and thus the claim follows.
\end{proof}

The last ingredient in our derivation of the uniform distribution of the occupation time of spherical fBM, Theorem~\ref{thm:sfbm_sojourn_time}, is a fluctuation result on random walk bridges.
We construct a random walk bridge $S_0= 0, S_1, \ldots, S_m, S_{m+1}=0$ from the (cyclically exchangeable) increments in~\eqref{eq:sfbm_increments} by setting
\begin{align*}
  S_k &\coloneqq \sum_{\ell=1}^{k} (X_{U_{m:\ell}}-X_{U_{m:\ell-1}}) = X_{U_{m:k}},\qquad 1\leq k\leq m+1.
\end{align*}
The event $\{ X_{U_1}>0, \ldots, X_{U_m}>0  \}=\{X_{U_{m:1}}>0, \ldots, X_{U_{m:m}}>0\}$ may now be viewed as the event $\{S_1>0, \ldots, S_m>0\}$, i.e.\ that $(S_k)_{k=0}^{m+1}$ is positive (except for its two endpoints $S_0=S_{m+1}=0$). The fluctuation result on random walk bridges (with cyclically exchangeable increments) will be formulated quite generally and may be of independent interest. We stress that a similar result by Sparre Andersen \cite{Andersen1953b}, cf.~the exposition in~\cite[Section 1.3]{KabluchkoVysotskyZaporozhets2019}, is not sufficient for our purposes, as it assumes exchangeable increments rather than only \textit{cyclically} exchangeable increments. We refer to \cite{BergerBethencourt2023preprint} for similar results, but they do not apply to bridges, as needed in our case. Our notation partly follows the exposition in~\cite{KabluchkoVysotskyZaporozhets2019}.

Fix $m\in\mathbb N$. Let $\xi_1, \ldots, \xi_m$ be real r.v.s. Define the partial sums $(S_k)_{k=0}^m$ by
\begin{align*}
  S_0\coloneqq 0, \qquad S_k &\coloneqq \xi_1+\cdots+\xi_k,\qquad 1\leq k\leq m.
\end{align*}
We impose the following assumptions on the increments $\xi_1, \ldots, \xi_m$:
\begin{enumerate}[i)]
  \item Bridge property: $S_m=\xi_1+\ldots+\xi_m=0$ a.s.
  \item Cyclic exchangeability: For every cyclic permutation $\pi\in\Cyc(m)$ of $\{1, \ldots, m\}$ we have the distributional identity
    \begin{align*}
      (\xi_1, \ldots, \xi_m) =_d (\xi_{\pi(1)}, \ldots, \xi_{\pi(m)}).
    \end{align*}
  \item For any $1\leq k\leq m-1$ the distribution of $S_k$ has no atoms.
\end{enumerate}
We call $(S_k)_{k=0}^m$ a random walk bridge with cyclically exchangeable increments.

\begin{proposition}\label{prop:uniform_law_random_bridges}
Fix $m\in\mathbb N$. Let $(S_k)_{k=0}^m$ be a random walk bridge with cyclically exchangeable increments such that $S_k$ has no atoms for any $1\leq k\leq m$, then
  \begin{align*}
    \Prob{S_1>0, \ldots, S_{m-1}>0 } &= \frac{1}{m}.
  \end{align*}
\end{proposition}

\begin{proof} 
Define $\tau:=\min\{ j \in\{ 0,\ldots,m-1\} : S_j = \min_{k\in\{ 0,\ldots, m-1\}} S_k\}$ to be the index where the minimum of $S_0=0, S_1,\ldots, S_{m-1}$ is attained. 
We show that $\Prob{\tau=j}=\Prob{\tau=0}$ for all $j=0,1,\ldots,m-1$. Since $1=\sum_{j=0}^{m-1} \Prob{\tau=j}$, this will imply $\Prob{\tau=0}=\frac{1}{m}$. Noting further that $\Prob{\tau=0}=\Prob{S_1>0,\ldots,S_{m-1}>0}$ we will have proved our claim. In order to prove $\Prob{\tau=j}=\Prob{\tau=0}$ we will use the cyclic permutation $\pi$ given by 
$$
\pi(i):=\begin{cases} 
 i+j &: i=1,\ldots,m-j, \\ i-m+j &: i=m-j+1,\ldots,m.
\end{cases}
$$
Note that by cyclic exchangeability
$$
\Prob{\tau=0} = \Prob{ 0 < \sum_{i=1}^k \xi_i , k=1,\ldots,m-1} = \Prob{ 0 < \sum_{i=1}^k \xi_{\pi(i)} , k=1,\ldots,m-1}.
$$
We are going to analyse the conditions $0 < \sum_{i=1}^k \xi_{\pi(i)}$, $k=1,\ldots,m-1$, and see that they are equivalent to the event $\{\tau=j\}$.
Indeed, first note that for $k=1,\ldots,m-j-1$
$$
0 < \sum_{i=1}^k \xi_{\pi(i)} = \sum_{i=1}^k \xi_{i+j} = \sum_{i=j+1}^{k+j} \xi_{i} = S_{k+j}-S_j.
$$
This means that $S_j<S_\ell$ for all $\ell=j+1,\ldots,m-1$.
Second, note that for $k=m-j,\ldots,m-1$
\begin{eqnarray*}
0 &<& \sum_{i=1}^k \xi_{\pi(i)} =  \sum_{i=1}^{m-j} \xi_{\pi(i)}+ \sum_{i=m-j+1}^k \xi_{\pi(i)} =  \sum_{i=1}^{m-j} \xi_{i+j}+ \sum_{i=m-j+1}^k \xi_{i-m+j} \\
  &=&  \sum_{i=j+1}^{m} \xi_{i}+ \sum_{i=1}^{k-m+j} \xi_{i}=S_m - S_j + S_{k-m+j} = 0 - S_j + S_{k-m+j}.
\end{eqnarray*}
This means that $S_j < S_\ell$ for all $\ell=0,\ldots,j-1$. \end{proof}

We are now ready to prove Theorem~\ref{thm:sfbm_sojourn_time}.

\begin{proof}[Proof of Theorem~\ref{thm:sfbm_sojourn_time}]
  Recall that the uniform distribution on $(0, 1)$ has moment sequence $\int_0^1 x^m dx=\frac{1}{m+1}$, $m\in\mathbb N$. For this reason, according to Proposition~\ref{prop:sampling} it is sufficient to show that $\Prob{X_{U_1}>0, \ldots, X_{U_m}>0}=\frac{1}{m+1}$ for $m\geq 1$. From Proposition~\ref{prop:sfbm_increments} we know that we can view $\{X_{U_1}>0, \ldots, X_{U_m}>0\}=\{X_{U_{m:1}}>0, \ldots, X_{U_{m:m}}>0\}$ as the event $\{S_1>0, \ldots, S_m>0\}$, where $S_1, \ldots, S_{m+1}$ is a random walk bridge with exchangeable increments defined in~\eqref{eq:sfbm_increments}. The claim thus follows from Proposition~\ref{prop:uniform_law_random_bridges}.
\end{proof}

\subsection{Proofs for the results on L\'evy processes and L\'evy bridges}
\subsubsection{L\'evy processes}

For the proof of Theorem~\ref{thm:occupation_time} we rely on some well-known results of Spitzer and some basic facts on Bell polynomials that we now recall. Let $\xi_1, \xi_2, \ldots$ denote a sequence of i.i.d.~real-valued random variables with partial sums $S_n\coloneqq \sum_{k=1}^n \xi_k$, $n\in\mathbb N$. As a consequence of what is now called Spitzer's identity, he obtains the following fact. 

\begin{corollary}[Corollary 2 in~\cite{Spitzer1956}, Theorem 1 in~\cite{Andersen1954}]
The survival probabilities of the partial sums $S_1, S_2, \ldots$ have generating function
\begin{align}\label{eq:Spitzer}
  \sum_{k=0}^\infty t^k \Prob{ S_1\geq 0, \ldots, S_k\geq 0 } &= \exp\left ( \sum_{k=1}^\infty \frac{t^k}{k}\Prob{S_k\geq 0}  \right),\qquad |t|<1.
\end{align}
This identity still holds when the inequalities in~\eqref{eq:Spitzer} are replaced by strict inequalities.
\end{corollary}

Let us rewrite this identity in a more combinatorial form that is better suited for our purposes. To this end we utilize the Bell polynomials. For any two sequences of real numbers $v_\bullet=(v_k)_{k\in\mathbb{N}}$ and $w_\bullet=(w_k)_{k\in\mathbb N}$ let
\[
B_k(v_\bullet, w_\bullet)\coloneqq \sum_{\ell=1}^k v_\ell B_{k,\ell}(w_\bullet),\quad k\in\mathbb N,
\]
denote the $k$-th complete Bell polynomial (associated with $(v_\bullet, w_\bullet)$),
where
\[ B_{k,\ell}(w_\bullet)\coloneqq \sum_{\rho\in{\mathscr P}_{k,\ell}}\prod_{B\in\rho}w_{\card B},\quad 1\leq \ell\leq k,\]
denotes the $(i,\ell)$-th partial Bell polynomial (associated with $w_\bullet$) and ${\mathscr P}_{i,\ell}$ denotes the set of all partitions of $\{1, \ldots, i\}$ that contain exactly $\ell$ blocks. We use the well known fact, cf.~\cite[Equation (1.11)]{Pitman2006}, that for any two sequences $v_\bullet, w_\bullet$ the exponential generating function of the associated complete Bell polynomials is given by
	\begin{align}\label{eq:bell_generating_function}
	\sum_{k=1}^\infty B_{k}(v_\bullet, w_\bullet)\frac{x^k}{k!} = v(w(x)),
	\end{align}
	whenever either of these quantities is well-defined and where $v(x)\coloneqq \sum_{k\geq 1} v_k \frac{x^k}{k!}$ and $w(y)\coloneqq \sum_{k\geq 1}w_k \frac{y^k}{k!}$ denote the exponential generating function of $v_\bullet=(v_k)_{k\in\mathbb{N}}$ and $w_\bullet=(w_k)_{k\in\mathbb{N}}$, respectively. For more information on Bell polynomials, the interested reader is referred to the lecture notes~\cite{Pitman2006}.\smallskip

We will work with the following reformulation of Spitzer's result: 
\begin{corollary}\label{cor:survival_probabilities}
For the survival probability of the sequence of partial sums we obtain
\begin{align}\label{eq:survival_probability}
  \Prob{ S_1\geq 0, \ldots, S_k\geq 0 } &= \frac{1}{k!}\, \sum_{\rho\in{\mathscr P}_k} \prod_{B\in \rho} (\card B-1)!\, \Prob{S_{\card B}\geq 0};
\end{align}
and the identity in~\eqref{eq:survival_probability} still holds when the inequalities are replaced by strict inequalities.
\end{corollary}
\begin{proof}
    Define the sequences $v_\bullet\coloneqq (v_k)_{k\in\mathbb N}$ and $w_\bullet\coloneqq (w_k)_{k\in\mathbb N}$ by setting 
    \begin{align*}
        v_k= 1\quad\text{and}\quad w_k= (k-1)!\, \Prob{S_k\geq 0},\quad k\geq 1.
    \end{align*}
    With this particular choice for $v_\bullet, w_\bullet$, we find that $v(x)=e^x-1$ and \linebreak $w(x)=\sum_{k=1}^\infty \frac{x^k}{k}\Prob{S_k\geq 0}$ and we can observe that
    the right hand side in~\eqref{eq:Spitzer} equals $v(w(t))+1$. For $v(w(t))$ we can use the expansion~\eqref{eq:bell_generating_function}. Comparing this to the left hand side of~\eqref{eq:Spitzer} we find that
\begin{align*}
  \Prob{ S_1\geq 0, \ldots, S_k\geq 0 } &= \frac{1}{k!}\, B_k(v_\bullet, w_\bullet) 
  \\
  &= \frac{1}{k!}\, \sum_{\ell=1}^k B_{k,\ell}(w_\bullet)
  \\
  &= \frac{1}{k!}\, \sum_{\ell=1}^k \sum_{\rho\in{\mathscr P}_{k,\ell}} \prod_{B\in\rho} w_{\card B}
  \\
    &= \frac{1}{k!}\, \sum_{\rho\in{\mathscr P}_{k}} \prod_{B\in\rho} (\card B-1)!\, \Prob{ S_{\card B} \geq 0},
\end{align*}
as claimed, where we used that ${\mathscr P}_{k}$ is the disjoint union of the ${\mathscr P}_{k,\ell}$ for $\ell=1,\ldots,k$. The claim for strict inequalities follows from the same proof using strict inequalities at all places.
\end{proof}

We are now prepared to prove Theorem~\ref{thm:occupation_time}. The proof idea is similar to our proof of the Brownian (or stable L\'evy) arcsine law, Theorem~\ref{thm:levysfirstarcsinelaw}. The moments are rewritten as persistence probabilities at random times that come from the normalised jump-times of an independent Poisson process. As we are not assuming the scaling property we cannot scale out the terminal time. This forces us to work with a L\'evy process extension of the sampling formula at Poisson times. Combined with the above variant of Spitzer's identity the claim can be deduced.
\begin{proof}[Proof of Theorem~\ref{thm:occupation_time}]

For the proof it is convenient to turn to a variant of the sampling method. Specifically, instead of starting with the $m$-th moment of the sojourn time of $X$ up until time $t$, we focus instead on its Laplace transform. We use the Poisson sampling formula
\begin{align}\label{eqn:F}
    F(q) &\coloneqq \int_0^\infty e^{-qt} \Ex[A_t^m]dt = \frac{m!}{q^{m+1}}\Prob{X_{T_1^{(q)}}>0, \ldots, X_{T_m^{(q)}}>0},\qquad m\in\mathbb N, q>0,
\end{align}
where $(T_k^{(q)})$ denotes the sequence of waiting times in a standard Poisson process of intensity $q>0$ independent of $X$. The formula was for instance used by  \cite{GetoorSharpe1994} but most certainly also appeared elsewhere. Here is a quick proof for completeness. First note that the Markov property of the random walk $X_{T_1^{(q)}}, X_{T_2^{(q)}}, \ldots$ yields
\begin{align*}
    \Prob{X_{T_1^{(q)}}> 0, \ldots, X_{T_m^{(q)}}> 0}
    &= \int_{x_1> 0}\cdots \int_{x_m> 0}P(dx_1)\cdots P(dx_m-x_{m-1})\\
    &=q^m \int_{x_1>0}\cdots \int_{x_m>0} U^q(dx_1)\cdots U^q(dx_m-x_{m-1}),
\end{align*}
where $P(A):=\Prob{X_{T_1^{(q)}}\in A}$ is the jump distribution of the random walk and the quantity $U^q(A):=\int_0^\infty e^{-q t} \Prob{X_t\in A} dt$ is the so-called $q$-potential measure of $X$. Here we used that $P=q U^q$. To rewrite $F$  denote by $\bar X$ the killed process, i.e. the L\'evy process killed at an independent exponential time with parameter $q$. By $\bar p$ denote the transition kernel of the killed process. Then Fubini, monotone convergence, the Markov property, and symmetry of the integrand yield
\begin{align*}
   &\quad \int_0^\infty e^{-q t} \mathbb E[A_t^m]dt\\
   &=\frac{1}{q}\int_0^\infty\cdots \int_0^\infty \mathbb E[\mathbf 1_{\bar X_{s_1}> 0}\cdots \mathbf 1_{\bar X_{s_m> 0}}]ds_m\cdots ds_s \\
    &=\frac{1}{q}\lim_{N_1,...,N_m\to\infty}\int_0^{N_1}\cdots \int_0^{N_m} \mathbb E[\mathbf 1_{\bar X_{s_1}> 0}\cdots \mathbf 1_{\bar X_{s_m> 0}}]ds_m\cdots ds_1 \\
 &=\frac{m!}{q}\lim_{N_1,...,N_m\to\infty}\int_{0}^{N_1}\int_{s_1}^{N_2}\cdots \int_{s_{m-1}}^{N_m} \mathbb E[\mathbf 1_{\bar X_{s_1}> 0}\cdots \mathbf 1_{\bar X_{s_m> 0}}]ds_m\cdots ds_1 \\
  &=\frac{m!}{q}\lim_{N_1,...,N_m\to\infty}\int_{0}^{N_1}\int_{s_1}^{N_2}\cdots \int_{s_{m-1}}^{N_m} \\
    &\qquad\times \int_{x_1>0}\cdots \int_{x_m>0} \bar p_{s_1}(dx_1)\cdots \bar p_{s_m-s_{m-1}}(dx_m-x_{m-1})ds_m\cdots ds_1 \\
    &=\frac{m!}{q}\lim_{N_1,...,N_m\to\infty}\int_{0}^{N_1}\int_{0}^{N_2-s_1}\cdots \int_{0}^{N_m-s_{m-1}} \\
    &\qquad\times \int_{x_1>0}\cdots \int_{x_m>0} \bar p_{s_1}(dx_1)\cdots \bar p_{s_m}(dx_m-x_{m-1})ds_m\cdots ds_1 \\
    &=\frac{m!}{q}
    \int_{x_1>0}\cdots \int_{x_m>0} U^q(dx_1)\cdots U^q(dx_m-x_{m-1}).
\end{align*}
Combining the two previous displays yields \eqref{eqn:F}. 
Combined with  Corollary~\ref{cor:survival_probabilities} we obtain
\begin{align*}
    F(q) &= \frac{1}{q^{m+1}}  \sum_{\rho\in{\mathscr P}_m} \prod_{B\in\rho} (\card B-1)! \, \Prob{S_{\card B}>0},
\end{align*}
with $S_k\coloneqq X_{T_k^{(q)}}$. Since $T_k^{(q)}$ is gamma distributed with parameters $k, q$; i.e.\ with density $s\mapsto q^k s^{k-1}e^{-qs}/(k-1)!\mathbf 1\{ s>0\}$, and setting $p_s\coloneqq \Prob{X_s>0}$ we conclude
\begin{align*}
    F(q) &= \frac{1}{q}  \sum_{\rho\in{\mathscr P}_m} \prod_{B\in\rho} \int_0^\infty s^{\card B-1}p_se^{-qs}ds.
\end{align*}
The integral on the right-hand side is a Laplace transform, and we will denote the Laplace transform of a function $f : [0,\infty) \to \mathbb R$ by $(\mathcal L f)(q):=\int_0^\infty f(s) e^{-qs}d s$. Using basic properties of Laplace transforms we can thus write
\begin{align*}
    F(q) &= \frac{1}{q}  \sum_{\rho\in{\mathscr P}_m} \prod_{B\in\rho} \mathcal L\left (s^{\card B-1}p_s\right )(q)\\
      &= \frac{1}{q}    \mathcal L\left (\sum_{\rho\in{\mathscr P}_m} \conv_{B\in\rho} s^{\card B-1}p_s\right )(q)\\
      &= \mathcal L\left (1\ast \sum_{\rho\in{\mathscr P}_m} \conv_{B\in\rho} s^{\card B-1}p_s\right )(q).
\end{align*}
From this calculation we find
\begin{align*}
  \Ex[A_t^m] &= \left (1\ast\sum_{\rho\in{\mathscr P}_m} \conv_{B\in\rho} s^{\card B-1}p_s\right )(t),
\end{align*}
which is the claim.
\end{proof}

Before we turn to the proof of Corollary~\ref{cor:GS} we provide a helpful lemma.

\begin{lemma}\label{lem:monomials_convolution}
    Fix $m\in\mathbb N$ and positive real numbers $a_1, \ldots, a_m>0.$ Then, for $t>0$,
    \begin{align*}
        \left( \conv_{k=1}^m \left(s^{a_k-1}\right)\right) (t) = \Gamma\left(\sum_{k=1}^m a_k\right)^{-1} \prod_{k=1}^m \Gamma(a_k) \cdot t^{\sum_{k=1}^m a_k-1}.
    \end{align*}
\end{lemma}
\begin{proof}
    We show the claim by induction on $m$. It is clear that the claim holds for $m=1.$ Assume now the claim is true for some positive integer $m\in\mathbb N.$ Then, using the induction hypothesis in the second equality,
    \begin{align*}
        \left( \conv_{k=1}^{m+1} \left(s^{a_k-1}\right)\right ) (t) &= \left ( \conv_{k=1}^{m} \left(s^{a_k-1}\right) \ast s^{a_{m+1}-1}  \right) (t)\\
        &= \left ( \Gamma\left (\sum_{k=1}^m a_k\right )^{-1} \prod_{k=1}^m \Gamma(a_k) \cdot s^{\sum_{k=1}^m a_k-1} \ast s^{a_{m+1}-1}  \right ) (t)\\
        &= \Gamma \left(\sum_{k=1}^m a_k\right)^{-1} \prod_{k=1}^m \Gamma(a_k)  \int_0^t (t-s)^{\sum_{k=1}^m a_k-1} s^{a_{m+1}-1}ds\\
        &= \Gamma \left(\sum_{k=1}^m a_k\right)^{-1} \prod_{k=1}^m \Gamma(a_k) \cdot  t^{\sum_{k=1}^{m+1} a_k-1}\int_0^1 (1-s)^{\sum_{k=1}^m a_k-1} s^{a_{m+1}-1}ds\\
        &= \Gamma \left(\sum_{k=1}^m a_k\right)^{-1} \prod_{k=1}^m \Gamma(a_k) \cdot t^{\sum_{k=1}^{m+1} a_k-1}\frac{\Gamma(\sum_{k=1}^m a_k)\Gamma(a_{m+1})}{\Gamma(\sum_{k=1}^{m+1} a_k)}\\
        &= \Gamma \left(\sum_{k=1}^{m+1} a_k\right)^{-1} \prod_{k=1}^{m+1} \Gamma(a_k) \cdot t^{\sum_{k=1}^{m+1} a_k-1},
    \end{align*}
where we transformed coordinates in the integral in the fourth equality and used the beta integral in the fifth equality.
\end{proof}

\begin{proof}[Proof of Corollary \ref{cor:GS}]
  Assume that $\Prob{X_t>0}=c\in (0, 1)$ for all $t>0$. By~\eqref{eq:occupation_time} we have
  \begin{align*}
      \Ex[A_t^m] = \sum_{\rho\in{\mathscr P}_m}  \int_0^t \conv_{B\in\rho} \left (u^{\card B-1}\Prob{X_u>0}\right )(s)ds
    = \sum_{\rho\in{\mathscr P}_m} c^{\card\rho} \int_0^t \conv_{B\in\rho} \left (u^{\card B-1}\right )(s)ds
    \end{align*}
    and applying Lemma~\ref{lem:monomials_convolution} the right hand side equals
    \begin{align*}
    &\quad \frac{1}{\Gamma(m)}\sum_{\rho\in{\mathscr P}_m} c^{\card\rho}\prod_{B\in\rho}(\card B-1)! \cdot \int_0^t s^{m-1}ds\\
    &= \frac{t^{m}}{m!}\sum_{\rho\in{\mathscr P}_m} c^{\card\rho}\prod_{B\in\rho}(\card B-1)!\\
    &= \frac{t^{m}}{m!}\sum_{b=1}^m c^b \sum_{\rho\in{\mathscr P}_{m, b}} \prod_{B\in\rho}(\card B-1)!.
  \end{align*}
Notice that $(k-1)!$ is the number of cyclic permutations of $k$ elements, thus \linebreak $\sum_{\rho\in{\mathscr P}_{m, b}} \prod_{B\in\rho}(\card B-1)!$ is the number of permutations of $\{1,\ldots,m\}$ with $b$ blocks, also known as the $(n, b)$-th unsigned Stirling number, which we denote by $\stirling{n}{k}$. 
Recall that the unsigned Stirling numbers may also be written as
\begin{align*}
    \risfac{c}{m} = \sum_{b=0}^m \stirling{m}{k} c^b,
\end{align*}
where $\risfac{c}{m}\coloneqq c(c+1)\cdots (c+m-1)$ and we recall that 
$\stirling{m}{0}=0$ if $m>0$. Putting everything together, we conclude that $\Ex[A_t^m]=t^m\risfac{c}{m}/m!$, which is the $m$-th moment of $tA$ where the distribution of $A$ is the arcsine law on $(0,1)$ with parameter $c$. Since this distribution is uniquely determined by its moments, we are done with the proof of the first implication.\smallskip

To see the opposite implication, assume that $t^{-1} A_t$ is arcsine distributed with parameter $c$. In particular, the first moment has to have the form $\mathbb E[A_t]=c t$ for all $t>0$. Using the first moment formula (\ref{eq:occupation_time_moments}), we know that the positivity function of the L\'evy process is constant.
\end{proof}

\subsubsection{L\'evy bridges} 
We now offer an elementary derivation of the occupation time distribution of L\'evy bridges. In some sense the proof is a simpler version of our proof for the theorem on the occupation time of spherical Brownian motion. Here the proof can rely on exchangeability whereas the spherical situation is more subtle and requires to work with cyclic exchangeability only. The proof is based on the so-called Baxter's combinatorial lemma for permutations of vectors. Fix $n\in\mathbb N$ and $z_1, \ldots, z_n\in\mathbb R^2$. For any permutation $\pi\in\mathfrak \Cyc(m)$ define the corresponding partial sums
We now offer an elementary derivation of the occupation time distribution of L\'evy bridges. In some sense the proof is a simpler version of our proof for the theorem on the occupation time of spherical Brownian motion. Here the proof can rely on exchangeability whereas the spherical situation is more subtle and requires to work with cyclic exchangeability only. The proof is based on the so-called Baxter's combinatorial lemma for permutations of vectors. Fix $n\in\mathbb N$ and $z_1, \ldots, z_n\in\mathbb R^2$. For any permutation $\pi\in\mathfrak \Cyc(m)$ define the corresponding partial sums
\begin{align*}
  s_0[\pi]\coloneqq 0, \,\,\, s_k[\pi]\coloneqq\sum_{\ell=1}^k z_{\pi(\ell)},\quad 1\leq k\leq n.
\end{align*}
For the sake of brevity we set $s_k\coloneqq s_k[id_n]=\sum_{\ell=1}^k z_\ell$ for $0\leq k\leq n$, where $id_n\colon [n]\to [n]$ denotes the identity permutation $id_n(k)=k$. Notice that $s_n[\pi]=\sum_{\ell=1}^n z_{\pi(\ell)}=s_n$ does not depend on $\pi$. Moreover, for any subset $M\subseteq \{1,\ldots,n\}$ let $s_M\coloneqq \sum_{k\in M}z_k$. Following Baxter, we call $z_1, \ldots, z_n$ skew if the fact that $z_M$ and $z_{M'}$ lie on a common line (i.e.~there exists a real $c\neq 0$ such that $z_M=cz_{M'}$) implies $M=M'$. Any point $z\in\mathbb R^2\setminus \{0\}$ together with the origin $0\in\mathbb R^2$ defines a line in the plane containing $0$ and $z$ that divides the plane into two half spaces. We call these the left and right half space (in clockwise orientation when the direction of the line is induced by moving from $0$ to $z$) induced by $z$. Let $H(z)$ denote the left half space induced by $z$ and including the line containing $z$. Then Baxter's combinatorial lemma may be stated as follows.
\begin{lemma}[Baxter's combinatorial lemma, cf.~Lemma 1 in~\cite{Baxter1961}]\label{lem:baxters_lemma}
  Fix $n\in\mathbb N$ and assume that $z_1, \ldots, z_n\in\mathbb R^2$ that are skew. Then
  \begin{align*}
    \card\left \{\pi\in\Cyc(n)\colon \{s_k[\pi]\}_{k=1}^n\subseteq H(s_n) \right \} &=1.
  \end{align*}
  In words, there is precisely one cyclic permutation $\pi$ of $z_1, \ldots, z_n$ such that the corresponding partial sums $s_1[\pi], \ldots, s_n[\pi]$ all lie in the left half space $H(s_n)$ of $s_n$.
\end{lemma}
We note that Baxter's lemma is rather elementary to prove, the proof is a clever few line computation. The way in which we will apply Baxter's lemma is the following. If the partial sums $(s_k)_{k=1}^n$ induced by skew points $z_1, \ldots, z_n\in\mathbb R^2$ are such that $s_n$ lies on the positive $x$-axis, then $H(s_n)$ is the upper half-plane. Thus Baxter's lemma is well suited to approach persistence probabilities from a combinatorial perspective. 

Moreover, we make the following observation. Consider a real function $f:\mathbb R\to\mathbb R$ such that $f(0)=0$ and fix $0=t_0<t_1<\cdots<t_n<t_{n+1}=1$. Define the function $\mathring{f}\colon [0, 1]\to\mathbb R$ by setting $\mathring{f}(t)\coloneqq f(t)-tf(1)$, and call $\mathring{f}$ the bridge induced by $f$. 

\begin{lemma}
\label{lem:baxter_corollary}
We have $\mathring{f}(t_k)>0$ for all $1\leq k\leq n$ if and only if the points $(t_k, f(t_k))$, $1\leq k\leq n+1$, all lie in the left half plane induced by $(1, f(1))$ (which is $H((1,f(1))$).
\end{lemma}

\begin{proof} The line through the origin containing $(1, f(1))$ may be parameterised as $\{ (t, t f(1))\colon t\in\mathbb R  \}$. Inserting the arguments $t_k$ gives the claim. 
\end{proof}

We are now ready to prove Theorem~\ref{thm:occupation_time_levy_bridge}. Proposition~\ref{prop:sampling} reduces the problem to persistence probabilities of L\'evy bridges, which is then reformulated with the help of Lemma~\ref{lem:baxter_corollary}. Baxter's lemma then simplifies the expressions to the moments of the uniform distribution.
\begin{proof}[Proof of Theorem \ref{thm:occupation_time_levy_bridge}]
Recall that the uniform distribution on $(0, 1)$ is uniquely identified by its moment sequence $\int_0^1 x^m dx = \frac{1}{m+1}$, $m\geq 1$. By Proposition~\ref{prop:sampling} it suffices to show that, for any $m\in\mathbb N$,
  \begin{align*} 
      \Prob{\mathring X_{U_{m:1}}>0, \ldots, \mathring X_{U_{m:m}}>0} = \frac{1}{m+1},
  \end{align*}
  where $U_1, U_2, \ldots$ is an i.i.d.\ sequence of uniform $(0, 1)$ random variables independent of $(X_t)$ and $U_{m:1}\leq \ldots\leq U_{m:m}$ is the corresponding order statistics. We further set $U_{m:0}:=0$ and $U_{m:m+1}:=1$. 
  
  \textit{Step 1:} We first show that the random vector $(X_{U_{m:k}}- X_{U_{m:k-1}})_{k=1}^{m+1}$ is exchangeable. For this it suffices to show that for any permutation $\pi$ of $\{1, \ldots, m+1\}$ and for any $t_1, \ldots, t_{m+1}\in\mathbb R$
  \begin{align} \label{eqn:levyincone2}
    \Ex\left [\exp\left (i\sum_{k=1}^{m+1} t_k(X_{U_{m:k}}-X_{U_{m:k-1}})\right )\right ] &= \Ex\left [\exp\left (i\sum_{k=1}^{m+1} t_{k}(X_{U_{m:\pi(k)}}-X_{U_{m:\pi(k)-1}})\right )\right ].
  \end{align}
Let $B\coloneqq \{ x\in\mathbb R\colon \inorm x\leq 1\}$ denote the unit ball in $\mathbb R$, and let $(a, \gamma, \nu)$ denote the generating triplet of the law of $X_1$, where $a\geq 0$, $\gamma\in\mathbb R$, and $\nu$ is a measure on $\mathbb R$ with $\nu(\{ 0\})=0$ and $\int_{\mathbb R} (\abs x\wedge 1)\nu(dx)<\infty$. Conditionally given $U\coloneqq (U_1, \ldots, U_{m})$, and using the fact that $X$ has independent increments, we have
\begin{align}
    &\quad \Exc{ \exp\left (i\sum_{k=1}^{m+1} t_k(X_{U_{m:k}}-X_{U_{m:k-1}})\right )}{U}
    \notag 
    \\
    &= \prod_{k=1}^{m+1} \Exc{\exp(i t_k(X_{U_{m:k}}-X_{U_{m:k-1}}))}{U}
    \notag
    \\
    &= \prod_{k=1}^{m+1} \exp\left ((U_{m:k}-U_{m:k-1}) \left (-\frac 1 2 t_k^2 a + i\gamma t_k + \int_\mathbb R (e^{it_k x}-1-izx\mathbf 1_{B}(x)) \nu(dx) \right )\right ), \label{eqn:levybridgetwo}
\end{align}
where in the last step we used the well-known L\'evy-Khinchine representation of an infinitely divisible distribution, cf.~\cite[Theorem 8.1]{Sato1999}. Using the fact that the $m+1$ gaps $(U_{m:1}-U_{m:0}, U_{m:2}-U_{m:1}, \ldots, U_{m:m+1}-U_{m:m})$ induced by $U_1, \ldots, U_m$ obey a Dirichlet distribution with parameters $1, \ldots, 1$, and thus constitute an exchangeable random vector, we obtain from (\ref{eqn:levybridgetwo}) that
\begin{align*}
    & \Ex\,\Exc{ \exp\left (i\sum_{k=1}^{m+1} t_k(X_{U_{m:k}}-X_{U_{m:k-1}})\right )}{U} \\
    &= \Ex\,\Exc{ \exp\left (i\sum_{k=1}^{m+1} t_k(X_{U_{m:\pi(k)}}-X_{U_{m:\pi(k)-1}})\right )}{U}.
\end{align*}
By Fubini's theorem, this shows (\ref{eqn:levyincone2}). We can now come to the main argument of the proof.

\textit{Step 2:} Set
$$
S_k:=\left(U_{m:k},X_{U_{m:k}}\right) = \sum_{i=1}^k \left( U_{m:i}-U_{m:i-1}, X_{U_{m:i}}-X_{U_{m:i-1}}\right),\quad k=0,\ldots,m+1.
$$
Note that the events $\{ (U_{m:1}, X_{U_{m:1}}), \ldots, (U_{m:m}, X_{U_{m:m}}) \in H((1,X_1)) \} = \{ S_1, \ldots, S_m\in H((1,X_1)) \}$ and $\{ \mathring X_{U_{m:1}}>0, \ldots,  \mathring X_{U_{m:1}}>0 \}$ are equal by Lemma~\ref{lem:baxter_corollary}. Using the cyclic exchangeability in this first step (note that $X_1$ is not altered by the permutations), we obtain
\begin{align*}
    \Prob{S_1, \ldots, S_m\in H((1,X_1)) } &= \frac{1}{m+1} \sum_{\pi\in\Cyc(m+1)} \Prob{S_1[\pi], \ldots, S_m[\pi]\in H((1,X_1)) }\\
      &= \Ex\Big [  \frac{1}{m+1} \sum_{\pi\in\Cyc(m+1)} \mathbf 1\{S_1[\pi], \ldots, S_m[\pi]\in H((1,X_1)) \} \Big ]\\
      &= \frac{1}{m+1},
\end{align*}
where in the second to last line the sum equals one a.s.~by Baxter's combinatorial lemma. Here, we used that the points $(U_{m:k}-U_{m:k-1},X_{U_{m:k}}-X_{m:k-1})$ are almost surely skew in the application of Baxter's combinatorial lemma, which is due to the assumption that $X_1$ has no atoms.\end{proof}

\begin{remark}
    Note that the main argument applies to all stochastic processes whose increments over gaps induced by i.i.d.\ sampled times are exchangeable.
\end{remark}

\bibliographystyle{amsplain}
\bibliography{literature}

\providecommand{\bysame}{\leavevmode\hbox to3em{\hrulefill}\thinspace}
\providecommand{\MR}{\relax\ifhmode\unskip\space\fi MR }
\providecommand{\MRhref}[2]{%
  \href{http://www.ams.org/mathscinet-getitem?mr=#1}{#2}
}
\providecommand{\href}[2]{#2}
\begin{thebibliography}{10}

\bibitem{Aldous1993}
D.~Aldous, \emph{The continuum random tree. {III}}, Ann. Probab. \textbf{21}
  (1993), no.~1, 248--289.

\bibitem{Baxter1961}
G.~Baxter, \emph{A combinatorial lemma for complex numbers}, Ann. Math.
  Statist. \textbf{32} (1961), 901--904. \MR{126290}

\bibitem{Berger2010}
M.~Berger, \emph{Geometry revealed}, Springer, Heidelberg, 2010, A Jacob's
  ladder to modern higher geometry, Translated from the French by Lester
  Senechal. \MR{2724440}

\bibitem{BergerBethencourt2023preprint}
Q.~Berger and L.~B{\'e}thencourt, \emph{{An application of Sparre Andersen's
  fluctuation theorem for exchangeable and sign-invariant random variables}},
  {\it Séminaire de Probabilités}, to appear, 2023+.

\bibitem{Biane1991}
P.~Biane, \emph{Quantum random walk on the dual of {SU}{{\((n)\)}}}, Probab.
  Theory Relat. Fields \textbf{89} (1991), no.~1, 117--129.

\bibitem{Biane1992}
\bysame, \emph{Minuscule weights and random walks on lattices}, Quantum
  probability \& related topics, QP-PQ, vol. VII, World Sci. Publ., River Edge,
  NJ, 1992, pp.~51--65. \MR{1186654}

\bibitem{Bingham1971}
N.~H. Bingham, \emph{Limit theorems for occupation times of {M}arkov
  processes}, Z. Wahrscheinlichkeitstheorie und Verw. Gebiete \textbf{17}
  (1971), 1--22. \MR{281255}

\bibitem{BinghamDoney1988}
N.~H. Bingham and R.~A. Doney, \emph{On higher-dimensional analogues of the
  arc-sine law}, J. Appl. Probab. \textbf{25} (1988), no.~1, 120--131.
  \MR{929510}

\bibitem{Blumenson1960}
L.~E. Blumenson, \emph{Classroom {N}otes: {A} {D}erivation of
  {$n$}-{D}imensional {S}pherical {C}oordinates}, Amer. Math. Monthly
  \textbf{67} (1960), no.~1, 63--66. \MR{1530579}

\bibitem{Bousquet-MellouMishna2010}
M.~Bousquet-M{\'e}lou and M.~Mishna, \emph{Walks with small steps in the
  quarter plane}, Algorithmic probability and combinatorics. Papers from the
  AMS special sessions, Chicago, IL, USA, October 5--6, 2007 and Vancouver, BC,
  Canada, October 4--5, 2008, Providence, RI: American Mathematical Society
  (AMS), 2010, pp.~1--39.

\bibitem{DarlingKac1957}
D.~A. Darling and M.~Kac, \emph{On occupation times for {M}arkoff processes},
  Trans. Amer. Math. Soc. \textbf{84} (1957), 444--458. \MR{84222}

\bibitem{dembodinggao2013}
A.~Dembo, J.~Ding, and F.~Gao, \emph{Persistence of iterated partial sums},
  Ann. Inst. Henri Poincar\'{e} Probab. Stat. \textbf{49} (2013), no.~3,
  873--884. \MR{3112437}

\bibitem{DenisovWachtel2010}
D.~Denisov and V.~Wachtel, \emph{Conditional limit theorems for ordered random
  walks}, Electron. J. Probab. \textbf{15} (2010), 292--322, Id/No 11.

\bibitem{Desbois2007}
J.~Desbois, \emph{Occupation times for planar and higher dimensional {Brownian}
  motion}, J. Phys. A, Math. Theor. \textbf{40} (2007), no.~10, 2251--2262.

\bibitem{Devroye1986}
L.~Devroye, \emph{Nonuniform random variate generation}, Springer-Verlag, New
  York, 1986. \MR{836973}

\bibitem{Dyson1962}
F.~J. Dyson, \emph{A {Brownian}-motion model for the eigenvalues of a random
  matrix}, J. Math. Phys. \textbf{3} (1962), 1191--1198.

\bibitem{EichelsbacherKoenig2008}
P.~Eichelsbacher and W.~K{\"o}nig, \emph{Ordered random walks}, Electron. J.
  Probab. \textbf{13} (2008), 1307--1336.

\bibitem{Eldan2014}
R.~Eldan, \emph{Volumetric properties of the convex hull of an
  {{\(n\)}}-dimensional {Brownian} motion}, Electron. J. Probab. \textbf{19}
  (2014), 34, Id/No 45.

\bibitem{ErnstShepp2017}
P.~A. Ernst and L.~Shepp, \emph{On occupation times of the first and third
  quadrants for planar {B}rownian motion}, J. Appl. Probab. \textbf{54} (2017),
  no.~1, 337--342. \MR{3632623}

\bibitem{FayolleRaschel2012}
G.~Fayolle and K.~Raschel, \emph{Some exact asymptotics in the counting of
  walks in the quarter plane}, Proceeding of the 23rd international meeting on
  probabilistic, combinatorial, and asymptotic methods in the analysis of
  algorithms (AofA'12), Montreal, Canada, June 18--22, 2012, Nancy: The
  Association. Discrete Mathematics \& Theoretical Computer Science (DMTCS),
  2012, pp.~109--124.

\bibitem{FellerVol1}
W.~Feller, \emph{An introduction to probability theory and its applications.
  {V}ol. {I}}, third ed., John Wiley \& Sons, Inc., New York-London-Sydney,
  1968. \MR{228020}

\bibitem{FitzsimmonsGetoor1995}
P.~J. Fitzsimmons and R.~K. Getoor, \emph{Occupation time distributions for
  {L}\'{e}vy bridges and excursions}, Stochastic Process. Appl. \textbf{58}
  (1995), no.~1, 73--89. \MR{1341555}

\bibitem{GarbitRaschel2016}
R.~Garbit and K.~Raschel, \emph{On the exit time from a cone for random walks
  with drift}, Rev. Mat. Iberoam. \textbf{32} (2016), no.~2, 511--532.
  \MR{3512425}

\bibitem{GetoorSharpe1994}
R.~K. Getoor and M.~J. Sharpe, \emph{On the arc-sine laws for {L}\'{e}vy
  processes}, J. Appl. Probab. \textbf{31} (1994), no.~1, 76--89. \MR{1260572}

\bibitem{Hashorva}
E.~Hashorva, \emph{Boundary non-crossings of {B}rownian pillow}, J. Theoret.
  Probab. \textbf{23} (2010), no.~1, 193--208. \MR{2591910}

\bibitem{Istas2005}
J.~Istas, \emph{Spherical and hyperbolic fractional {B}rownian motion},
  Electron. Comm. Probab. \textbf{10} (2005), 254--262. \MR{2198600}

\bibitem{JohnsonMishnaYeats2018}
S.~Johnson, M.~Mishna, and K.~Yeats, \emph{A combinatorial understanding of
  lattice path asymptotics}, Adv. Appl. Math. \textbf{92} (2018), 144--163.

\bibitem{KabluchkoVysotskyZaporozhets2019}
Z.~Kabluchko, V.~Vysotsky, and D.~Zaporozhets, \emph{A multidimensional
  analogue of the arcsine law for the number of positive terms in a random
  walk}, Bernoulli \textbf{25} (2019), no.~1, 521--548. \MR{3892328}

\bibitem{Kac1951}
M.~Kac, \emph{On some connections between probability theory and differential
  and integral equations}, Proceedings of the {S}econd {B}erkeley {S}ymposium
  on {M}athematical {S}tatistics and {P}robability, 1950, Univ. California
  Press, Berkeley-Los Angeles, Calif., 1951, pp.~189--215. \MR{45333}

\bibitem{KallianpurRobbins1953}
G.~Kallianpur and H.~Robbins, \emph{Ergodic property of the {B}rownian motion
  process}, Proc. Nat. Acad. Sci. U.S.A. \textbf{39} (1953), 525--533.
  \MR{56233}

\bibitem{KhoshnevisanPemantle2004}
D.~Khoshnevisan and R.~Pemantle, \emph{Sojourn times of {B}rownian sheet},
  Period. Math. Hungar. \textbf{41} (2000), no.~1-2, 187--194. \MR{1812805}

\bibitem{Knight1996}
F.~B. Knight, \emph{The uniform law for exchangeable and {L}\'{e}vy process
  bridges}, Ast\'{e}risque (1996), no.~236, 171--188, Hommage \`a P. A. Meyer
  et J. Neveu. \MR{1417982}

\bibitem{LeGall1993}
J.-F. Le~Gall, \emph{The uniform random tree in a {Brownian} excursion},
  Probab. Theory Relat. Fields \textbf{96} (1993), no.~3, 369--383.

\bibitem{Levy1939}
P.~L\'{e}vy, \emph{Sur certains processus stochastiques homog\`enes},
  Compositio Math. \textbf{7} (1939), 283--339. \MR{919}

\bibitem{Levy1965}
\bysame, \emph{Processus {S}tochastiques et {M}ouvement {B}rownien},
  Gauthier-Villars \& Cie, Paris, 1965, Suivi d'une note de M. Lo\`eve,
  Deuxi\`eme \'{e}dition revue et augment\'{e}e. \MR{190953}

\bibitem{MeyreWerner1995}
T.~Meyre and W.~Werner, \emph{On the occupation times of cones by {Brownian}
  motion}, Probab. Theory Relat. Fields \textbf{101} (1995), no.~3, 409--419.

\bibitem{MoertersPeres2010}
P.~M\"{o}rters and Y.~Peres, \emph{Brownian motion}, Cambridge Series in
  Statistical and Probabilistic Mathematics, vol.~30, Cambridge University
  Press, Cambridge, 2010, With an appendix by Oded Schramm and Wendelin Werner.
  \MR{2604525}

\bibitem{Mountford1990}
T.~S. Mountford, \emph{Limiting behaviour of the occupation of wedges by
  complex {Brownian} motion}, Probab. Theory Relat. Fields \textbf{84} (1990),
  no.~1, 55--65.

\bibitem{Pitman1999}
J.~Pitman, \emph{Brownian motion, bridge excursion, and meander characterized
  by sampling at independent uniform times}, Electron. J. Probab. \textbf{4}
  (1999), 33, Id/No 11.

\bibitem{Pitman2006}
\bysame, \emph{Combinatorial stochastic processes}, Lecture Notes in
  Mathematics, vol. 1875, Springer-Verlag, Berlin, 2006, Lectures from the 32nd
  Summer School on Probability Theory held in Saint-Flour, July 7--24, 2002,
  With a foreword by Jean Picard. \MR{2245368}

\bibitem{Sato1999}
K.~Sato, \emph{L\'{e}vy processes and infinitely divisible distributions},
  Cambridge Studies in Advanced Mathematics, vol.~68, Cambridge University
  Press, Cambridge, 1999, Translated from the 1990 Japanese original, Revised
  by the author. \MR{1739520}

\bibitem{Andersen1953b}
E.~Sparre~Andersen, \emph{On the fluctuations of sums of random variables},
  Math. Scand. \textbf{1} (1953), 263--285. \MR{58893}

\bibitem{Andersen1954}
\bysame, \emph{On the fluctuations of sums of random variables. {II}}, Math.
  Scand. \textbf{2} (1954), 195--223.

\bibitem{Spitzer1956}
F.~Spitzer, \emph{A combinatorial lemma and its application to probability
  theory}, Trans. Amer. Math. Soc. \textbf{82} (1956), 323--339. \MR{79851}

\end{thebibliography}

\end{document}